\documentclass[reqno, 11pt]{amsart}

\usepackage{amsthm}

\usepackage[dvipsnames]{xcolor}

\usepackage{tikz}
\usepackage{tikz-cd}
\usetikzlibrary{shapes,arrows,cd}
\usetikzlibrary{calc,arrows.meta,positioning}
\usetikzlibrary{decorations.pathreplacing}

\usepackage{mathrsfs,amssymb}

\usepackage{pdfsync}
\usepackage{enumitem}

\usepackage{graphicx}

\usepackage{amsfonts} 
\usepackage{amsmath}
\usepackage{latexsym}

\usepackage{caption}

\usepackage[outline]{contour}
\contourlength{1.2pt}

\usepackage[
breaklinks=true,colorlinks=true,
linkcolor=black,urlcolor=black,citecolor=black,
bookmarks=true,bookmarksopenlevel=2]{hyperref}
\usepackage[capitalise, nameinlink, noabbrev]{cleveref}
\usepackage{nameref}

\usepackage[all]{xy}

\usepackage[square,sort,comma,numbers]{natbib}

\usepackage{url}
\usepackage{float}

\usepackage[textwidth=20mm]{todonotes}
\newtheorem{theorem}{Theorem}[section]

\newtheorem{proposition}[theorem]{Proposition}
\newtheorem{lemma}[theorem]{Lemma}

\theoremstyle{definition}

\newtheorem{remark}[theorem]{Remark}

\theoremstyle{plain}

\newtheoremstyle{named}{}{}{\itshape}{}{\bfseries}{.}{.5em}{\thmnote{#3}}
\theoremstyle{named}
\newtheorem*{namedtheorem}{Theorem}

\newtheoremstyle{named}{}{}{\itshape}{}{\bfseries}{.}{.5em}{\thmnote{#3}}
\theoremstyle{named}
\newtheorem*{namedclaim}{Claim}

\newcommand\qqed{\hfill\qed}
\newcommand\qqqed{\tag*{\qed}}

\makeatletter
\@namedef{subjclassname@2020}{\textup{2020} Mathematics Subject Classification}
\makeatother

\begin{document}

\noindent


\date{\today}

\title[Subsemigroups of monogenic free inverse semigroup]{On Finite Presentability of Subsemigroups of the Monogenic Free Inverse Semigroup}

\maketitle

\begin{center}
{\small
JUNG WON CHO%
\hspace{-.00em}\footnote{School of Mathematics and Statistics, University of St Andrews, St Andrews, Scotland,
UK. {\it Email: } {\tt jwc21@st-andrews.ac.uk}}
AND 
NIK RU\v{S}KUC
\hspace{-.20em}\footnote{School of Mathematics and Statistics, University of St Andrews, St Andrews, Scotland,
UK. {\it Email: }{\tt  nik.ruskuc@st-andrews.ac.uk. }{Supported by the Engineering and Physical Sciences Research Council [EP/V003224/1]}}
}
\end{center}

\begin{abstract}
The monogenic free inverse semigroup $FI_1$ is not finitely presented as a semigroup due to the classic result by Schein (1975). We extend this result and prove that a finitely generated subsemigroup of $FI_1$ is finitely presented if and only if it contains only finitely many idempotents. As a consequence, we derive that an inverse subsemigroup of $FI_1$ is finitely presented as a semigroup if and only if it is a finite semilattice.
\end{abstract}

\section{Introduction}
\label{sec:intro}

A semigroup $S$ is \emph{finitely presented} if it can be defined by a presentation $\langle A\mid R\rangle$ with both $A$ and $R$ finite.
Examples of finitely presented semigroups include:
all finite semigroups;
finitely generated free semigroups; finitely generated commutative semigroups \cite{redei};
all groups that are finitely presented as groups, including all finitely generated free groups; the bicyclic monoid \cite[Section 1.12]{CliffordPreston1}.
By way of contrast, Schein \cite{Schein1975} proved that free inverse semigroups are not finitely presented. The most surprising aspect of this, which is also the main step in the proof of the general result and the starting point for the present paper, is that this is true even for the monogenic free inverse semigroup:

\begin{namedtheorem}[Theorem 1 \protect{\normalfont(B.M. Schein \cite[Lemma 3]{Schein1975})}]
\label{thm:Schein}
The monogenic free inverse semigroup is not finitely presented as a semigroup.\qqed
\end{namedtheorem}

Motivated by Theorem~\ref{thm:Schein}, we investigate finite presentability of subsemigroups (not necessarily inverse) of the monogenic free inverse semigroup, and prove the following complete characterisation:

\begin{namedtheorem}[Main Theorem]
\label{thm:main}
A finitely generated subsemigroup of the monogenic free inverse semigroup is finitely presented as a semigroup if and only if it contains only finitely many idempotents.    
\end{namedtheorem}

As an immediate consequence we also prove:

\begin{namedtheorem}[Corollary]
\label{cor:inv_fp}
Let $S$ be a finitely generated inverse subsemigroup of the monogenic free inverse semigroup. Then $S$ is finitely presented as a semigroup if and only if $S$ is a finite semilattice.
\end{namedtheorem}

The two results are in a somewhat stark contrast with the behaviour of one-generator free objects in some other varieties.
For example, the non-trivial subgroups of the cyclic free group $\mathbb{Z}$ are all isomorphic to $\mathbb{Z}$, while subsemigroups of the monogenic free semigroup $\mathbb{N}$ are all finitely generated \cite{natural}, and hence finitely presented (\cite{redei}, \cite[Chapter 9]{CliffordPreston2}). This behaviour is also very different from finite presentability of inverse subsemigroups \emph{as inverse semigroups}: Oliveira and Silva proved in \cite{Oliveira2005} that every finitely generated inverse subsemigroup of the monogenic free inverse semigroup is finitely presented as an inverse semigroup; see also \cite{CR:paper1}.
Finally, and this time not surprisingly, the behaviour of subsemigroups of the monogenic free inverse semigroup is different from those of a free inverse semigroup of rank $>1$. In \cite[Example 6.1]{Oliveira2005} the authors exhibit an example of a finitely generated subsemigroup in the free inverse semigroup of rank $3$ which has no finite inverse semigroup presentation. Also in this direction, Reilly \cite{reilly72,reilly73} established a necessary and sufficient condition for an inverse subsemigroup of a free inverse semigroup given by a set of generators to be free over that set of generators.

There are further examples in literature of finiteness properties with respect to which (monogenic) free inverse semigroups show a radically different behaviour from their group- and `plain' semigroup counterparts.
For example, Cutting and Solomon \cite{CS:Rcf01} proved that the a free inverse semigroup does not admit a regular set of normal forms with respect to any generating set. Gray and Steinberg \cite{GS:Fim21} showed that free inverse semigroups do not satisfy the finiteness condition $\textup{FP}_2$. In both cases the brunt of the work is to deal with the monogenic case.

From a broader perspective, all the above  results, as well as the result that is proved here, can be viewed as 
confirming 
that the presentation theory of inverse semigroups is substantively different in nature from the corresponding theories for groups and semigroups.
We refer the reader to the survey articles \cite{Me:GSCC07,Me:Pisppa20} by Meakin for a nice treatment of the similarities, differences and resulting relationships. Perhaps the starkest expression of the difference is the recent result of Gray \cite{Gr:Uwp20} exhibiting an example of a one-relation inverse semigroup with an undecidable word problem;
and this can be contrasted to some extent with the more positive results such as
\cite{BMM:Wpim94,IMM:Orim01,GR:Gus}.

Our paper is structured as follows. 
In \cref{sec:prelim}, we give definitions and basic results about semigroup presentations and the monogenic free inverse semigroup. Sections \ref{sec:forward} and \ref{sec:backward} are devoted to proving each of the two implications in the \nameref{thm:main}. The \nameref{cor:inv_fp} is quickly proved at the end of the paper.

\section{Preliminaries}
\label{sec:prelim}

\subsection{General notation}
The set of natural numbers $\{1,2,\dots\}$ will be denoted by $\mathbb{N}$, the set of non-negative integers by 
$\mathbb{N}_0:=\mathbb{N}\cup\{0\}$, and the set of integers by $\mathbb{Z}$.
We will use the standard interval notations for subsets of $\mathbb{Z}$: $[m,n]:=\{m,m+1,\dots, n\}$, $[n]:=[1,n]$.

\subsection{Semigroup presentations}

Let $A$ be a set of letters and let $A^{+}$ be the free semigroup on $A$.  
A \emph{semigroup presentation} is an ordered pair $\langle A \mid R \rangle$; we call $A$ the \emph{generators} and $R \subseteq A^{+}\times A^{+}$ the \emph{defining relations} of $\langle A \mid R \rangle$. We write $R^{\sharp}$ for the congruence on $A^+$ generated by $R$. 
A semigroup $S$ is \emph{defined by $\langle A \mid R\rangle$ with respect to a mapping} $\theta \colon A \to S$ if the unique extension of $\theta$ to a homomorphism $\theta:A^+\rightarrow S$ is onto and $\ker\theta = R^{\sharp}$. 
We say that $S$ is \emph{finitely presented} if both $A$ and $R$ can be chosen to be finite.
Throughout this paper we will use the same symbol for a map $A\rightarrow S$ and for its unique homomorphism extension $A^+\rightarrow S$.

Suppose a semigroup $S$ is defined by $\langle A\mid R\rangle$ via $\theta\colon A\rightarrow S$.
Let $u, v\in A^+$. We say that \emph{$v$ is obtained from $u$ by applying a relation in $R$} if $u = prq$ and $v = psq$ where $p, q \in A^{\ast}$ and $(r,s)$ or $(s, r)$ is in $R$. An \emph{elementary sequence 
(with respect to $R$,  from $u$ to $v$)} is a sequence $w_1, \dots, w_n$ such that $w_1 = u$, $w_n = v$ and each $w_{i+1}$ is obtained from $w_i$ by applying a relation in $R$. Then $u\theta = v\theta$ if and only if there exists an elementary sequence from $u$ to $v$.

Now let $T$ be another semigroup, and $\phi: A\rightarrow T$ a mapping such that its extension $\phi \colon A^+\rightarrow T$ is onto.
We say that $T$ \emph{satisfies} the defining relations $R$ if for each $(r, s)\in R$ we have $r\phi = s\phi$, i.e. if $R \subseteq \ker\phi$. Then, there is a unique onto homomorphism $\psi \colon S \to T$ 
satisfying $\phi=\theta\psi$. In particular, $T$ is a homomorphic image of $S$. 
For a more systematic introduction to semigroup presentations, see \cite[Section 1.12]{CliffordPreston1}.

\subsection{Monogenic free inverse semigroups} 
For background on inverse semigroups we refer the reader to any standard textbook such as \cite{CliffordPreston1, Howie, Petrich, Lawson}.
It is well known that free objects exist for inverse semigroups; for an account see \cite[Chapter 6]{Lawson}.
In this paper, we focus solely on the monogenic free inverse semigroup, which will be denoted by $FI_1$. 
It can be represented as follows. The elements of $FI_1$ are certain triples:
\[
FI_1 := \bigl\{(-a, p, b)\in \mathbb{Z}^3 : a, b \geq 0, \, a+b >0, \, -a \leq p \leq b \bigr\},
\]
and the multiplication and inversion are as follows:
\[
\begin{aligned}
(-a_1, p_1, b_1)(-a_2, p_2, b_2) &= (-\max(a_1, a_2-p_1), p_1+p_2, \max(b_1, b_2+p_1)),\\
(-a, p, b)^{-1} &= (-(a+p), -p, b-p).
\end{aligned}
\]
An element $(-a, p, b)\in FI_1$ is an idempotent if and only if $p = 0$.

Some readers may be more used to viewing the elements in a free inverse semigroup as Munn trees; see \cite{Munn:Fis} and \cite[Chapter 6]{Lawson}.
For $FI_1$ the Munn trees are simply directed paths with inital and terminal vertices, and the translation between the two representations is straightforward. The triple $(-a,p,b)$ corresponds to the following Munn tree:

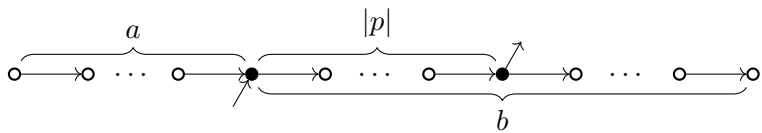
\begin{figure}[H]\centering

\begin{tikzpicture}[main/.style = {draw, circle, thick, inner sep = 1.5pt, minimum size=1pt}]
\node[main] (1) {};
\node[main] (2) [right = 0.8cm of 1] {};
\node[main] (3) [right = 2cm of 1] {};
\node[main, fill = black] (4) [right= 0.8cm of 3] {};
\node[main] (5) [right = 0.8cm of 4] {};
\node[main] (6) [right = 1.2cm of 5] {};
\node[main, fill=black] (7) [right = 0.8cm of 6] {};
\node[main] (8) [right = 0.8cm of 7] {};
\node[main] (9) [right = 1.2 of 8] {};
\node[main] (10) [right = 0.8cm of 9] {};

\node at ($(2)!.5!(3)$) {\ldots};
\node at ($(5)!.5!(6)$) {\ldots};
\node at ($(8)!.5!(9)$) {\ldots};

\draw[->] (1) -- (2);
\draw[->] (3) -- (4);
\draw[->] (4) -- (5);
\draw[->] (6) -- (7);
\draw[->] (7) -- (8);
\draw[->] (9) -- (10);
\draw[->] ($(4)-(60:5mm)$) -- (4);
\draw[->] (7) -- ($(7)+(60:5mm)$);

\draw [decorate,decoration={brace,amplitude=5pt,raise=5pt}] (1) -- (4) node[midway,above,yshift=10pt]{$a$};

\draw [decorate,decoration={brace,amplitude=5pt,raise=5pt}] (4) -- (7) node[midway,above,yshift=10pt]{$\vert p \vert$};

\draw [decorate,decoration={brace,amplitude=5pt,mirror, raise=5pt}] (4) -- (10) node[midway,above,yshift=-25pt]{$b$};

\end{tikzpicture}
\end{figure}

In the figure, the initial and terminal vertices are indicated by the `loose' in- and out-arrows respectively.
The sign of $p$ is positive if and only if the terminal vertex lies on the right hand side of the initial vertex. The numbers of edges between the initial and leftmost vertices, the initial and rightmost vertices, and the initial and terminal vertices correspond to $a$, $b$ and $\vert p \vert$, respectively.

We now give a semigroup presentation for $FI_1$.
\pagebreak

\begin{proposition}[\protect{\cite[Lemma 1]{Schein1975}}]
\label{pro:FI_1_present}
The monogenic free inverse semigroup $FI_1$ is defined by:
\[
\langle x, x^{-1} \mid \mathfrak{R} \rangle
\]
via $\phi\colon x\mapsto (0, 1, 1),\ x^{-1}\mapsto (-1, -1, 0)$,
with
\[
\mathfrak{R} := \bigl\{(xx^{-1}x, x), (x^{-1}xx^{-1}, x^{-1})\bigr\}\\
\cup
\bigl\{ (x^{-i}x^ix^jx^{-j}, x^{j}x^{-j}x^{-i}x^{i})\colon i, j \in \mathbb{N}_0,\ i+j>0)\bigr\}.\qqqed
\]
\end{proposition}

Note that, with the notation as above, we have  $(x^{-a}x^{a}x^{b}x^{-b}x^{p})\phi=(-a, p, b)$.

We partition $FI_1$ into its \emph{$\mathscr{D}$-classes}:
\[
D_n:= \bigl\{ (-a,p,b)\in FI_1\colon a+b=n\bigr\} \quad (n\in \mathbb{N}).
\]
These arise from the standard Green theory (see \cite[Chapter 2]{Howie}) but this will not be needed here, nor will any of the other Green's equivalences $\mathscr{L},\mathscr{R},\mathscr{H},\mathscr{J}$.
For every $n\in\mathbb{N}$, the union $\bigcup_{i\geq n} D_i$ is an ideal of $FI_1$.

\begin{figure}[H]
\begin{center}
\begin{tikzpicture}[scale=0.6]
\node [draw, shape = circle, fill = black, label=right:{\small $(0,0,1)$}, inner sep=-2, scale=0.5] (f1) at (0.5,8) {};
\node[draw, shape=circle, fill=black, label=left:{\small $(-1,0,0)$}, inner sep=-2, scale=0.5] (e1) at (-0.5,8) {}; 
\foreach \i in {2, ..., 3}{
    \node[draw, shape=circle, fill=black, label=right:{\small $(0,0,\i)$}, inner sep=-2, scale=0.5] (f\i) at ($(0.5,8)+(-60:\i-1)$) {};
    \node[draw, shape=circle, fill=black, label=left:{\small $(-\i, 0, 0)$}, inner sep=-2, scale=0.5] (e\i) at ($(-0.5,8)+(-120:\i-1)$) {};
}

\draw (f1)--(f3);
\draw (e1)--(e3);

\foreach \i in {1, ..., 3} {
    \draw ($(f1)+(-60:\i-1)$) -- ($(e3)+(0:\i)$);
    \draw ($(e1)+(-120:\i-1)$) -- ($(f3)-(0:\i)$);
}

\node [draw, shape = circle, fill = black, inner sep=-2, scale=0.2] at ($(e3)+(-120:0.75)$) {};
\node [draw, shape = circle, fill = black, inner sep=-2, scale=0.2] at ($(e3)+(-120:1)$) {};
\node [draw, shape = circle, fill = black, inner sep=-2, scale=0.2] at ($(e3)+(-120:1.25)$) {};

\node [draw, shape = circle, fill = black, inner sep=-2, scale=0.2] at ($(e3)!0.5!(f3) + (-90:0.75)$) {};
\node [draw, shape = circle, fill = black, inner sep=-2, scale=0.2] at ($(e3)!0.5!(f3) + (-90:1)$) {};
\node [draw, shape = circle, fill = black, inner sep=-2, scale=0.2] at ($(e3)!0.5!(f3) + (-90:1.25)$) {};

\node [draw, shape = circle, fill = black, inner sep=-2, scale=0.2] at ($(f3)+(-60:0.75)$) {};
\node [draw, shape = circle, fill = black, inner sep=-2, scale=0.2] at ($(f3)+(-60:1)$) {};
\node [draw, shape = circle, fill = black, inner sep=-2, scale=0.2] at ($(f3)+(-60:1.25)$) {};

\node [draw, shape = circle, fill = black, label = left:{\small $(-i,0,0)$}, inner sep=-2, scale=0.5] (ei1) at ($(e1)+(-120:4)$) {};
\node [draw,  inner sep=-2, scale = 0.2] (ei2) at ($(ei1)+(-120:1)$) {};
\node [draw,  inner sep=-2, scale = 0.2] (ei3) at ($(ei2)+(-120:1)$) {};
\node [draw, inner sep=-2, scale = 0.2] (ei4) at ($(ei3)+(-120:1)$) {};
\node [draw,  inner sep=-2, scale = 0.2] (ei5) at ($(ei4)+(-120:1)$) {};
\node [draw,  inner sep=-2, scale = 0.2] (ei6) at ($(ei5)+(-120:1)$) {};
\node [draw,  inner sep=-2, scale = 0.2] (ei7) at ($(ei6)+(-120:1)$) {};
\node [draw, inner sep=-2, scale=0.2] (ei8) at ($(ei7)+(-120:1)$) {};

\draw (ei1)--(ei8);

\node [draw, inner sep=-2, scale=0.2] (fi1) at ($(f1)+(-60:4)$) {};
\node [draw, shape = circle, fill = black, label=right:{\small $(0,0,j)$}, inner sep=-2, scale=0.5] (fi2) at ($(fi1)+(-60:1)$) {};
\node [draw, inner sep=-2, scale=0.2] (fi3) at ($(fi2)+(-60:1)$) {};

\foreach \i in {4, ...,8} {
    \node[draw, inner sep=-2, scale = 0.2] (fi\i) at ($(fi1)+(-60:\i-1)$) {};
}
\draw (fi1)--(fi8);

\draw (ei1) -- ($(ei8)+(0:7)$);
\draw (ei2) -- ($(ei8)+(0:6)$);
\draw (ei3) -- ($(ei8)+(0:5)$);
\draw (ei4) -- ($(ei8)+(0:4)$);
\draw (ei5) -- ($(ei8)+(0:3)$);
\draw (ei6) -- ($(ei8)+(0:2)$);
\draw (ei7) -- ($(ei8)+(0:1)$);

\draw (fi1) -- ($(fi1)+(-120:7)$);
\draw (fi2) -- ($(fi2)+(-120:6)$);
\draw (fi3) -- ($(fi3)+(-120:5)$);
\draw (fi4) -- ($(fi4)+(-120:4)$);
\draw (fi5) -- ($(fi5)+(-120:3)$);
\draw (fi6) -- ($(fi6)+(-120:2)$);
\draw (fi7) -- ($(fi7)+(-120:1)$);

\foreach \i in {1, ...,4} {
    \draw ($(ei1)+(0:\i)$) -- ($(ei1)+(0:\i)+(-120:7)$);
    \draw ($(fi1)-(0:\i)$) -- ($(fi1)-(0:\i)+(-60:7)$);
}
\node [draw, shape = circle, fill = black, inner sep=-2.5, scale = 0.5] (eifj) at ($(ei1)+(-60:6)$) {};
\node [black, fill=white, inner sep=2pt] at ($(eifj) + (2:1.5)$) {\small $(-i,0,j)$};

\node [draw, shape = circle, fill = black, inner sep=-2, scale=0.2] at ($(ei8)+(-120:0.5)$) {};
\node [draw, shape = circle, fill = black, inner sep=-2, scale=0.2] at ($(ei8)+(-120:0.75)$) {};
\node [draw, shape = circle, fill = black, inner sep=-2, scale=0.2] at ($(ei8)+(-120:1)$) {};

\node [draw, shape = circle, fill = black, inner sep=-2, scale=0.2] at ($(ei8)!0.5!(fi8) + (-90:0.5)$) {};
\node [draw, shape = circle, fill = black, inner sep=-2, scale=0.2] at ($(ei8)!0.5!(fi8) + (-90:0.75)$) {};
\node [draw, shape = circle, fill = black, inner sep=-2, scale=0.2] at ($(ei8)!0.5!(fi8) + (-90:1)$) {};

\node [draw, shape = circle, fill = black, inner sep=-2, scale=0.2] at ($(fi8)+(-60:0.5)$) {};
\node [draw, shape = circle, fill = black, inner sep=-2, scale=0.2] at ($(fi8)+(-60:0.75)$) {};
\node [draw, shape = circle, fill = black, inner sep=-2, scale=0.2] at ($(fi8)+(-60:1)$) {};

\end{tikzpicture}
\caption{Semilattice of idempotents of $FI_1$.}
\label{fig:semilattice}
\end{center}
\end{figure}
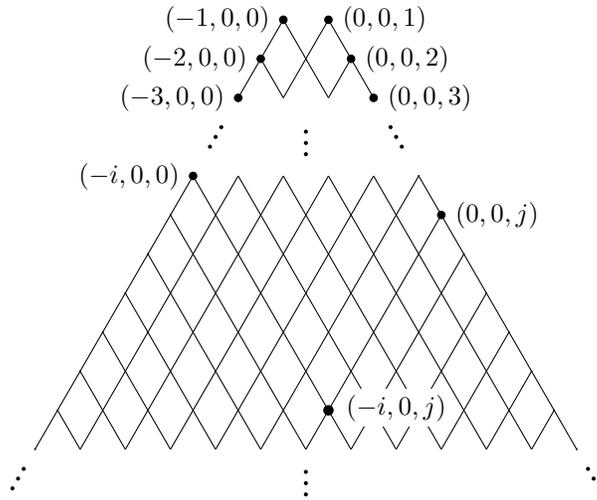

We say that an element $(-a, p, b)$ is \emph{positive, negative} or \emph{of sign} $0$ if $p > 0$, $p < 0$ or $p = 0$, respectively. The elements of sign $0$ are precisely the idempotents of $FI_1$. The set of all idempotents of $FI_1$ is partially ordered by 
\[
(-a, 0, b) \leq (-c, 0, d) \Leftrightarrow a \geq c \text{ and } b \geq d.
\]
This actually coincides with the natural partial order arising from general theory 
(see \cite[Chapter 5]{Howie} or \cite[Chapter 1]{Lawson}). The resulting partially ordered set is isomorphic to $\mathbb{N}_0\times \mathbb{N}_0\setminus \{(0,0)\}$ and is depicted in \cref{fig:semilattice}. In the diagram, the idempotents are in the same $\mathscr{D}$-class if and only if they are on the same horizontal level. Two idempotents $(-a, 0, b)$ and $(-c, 0, d)$ are \emph{incomparable} if and only if $(a-c)(b-d)<0$.

\section{Infinitely many idempotents}
\label{sec:forward}

The aim of this section is to prove the following proposition which is the forward direction of the \nameref{thm:main}.
\begin{proposition}
\label{prop:forward}
Let $S$ be a finitely generated subsemigroup of $FI_1$. If $S$ contains infinitely many idempotents, then $S$ is not finitely presented.
\end{proposition}

The proof is somewhat technical, but the underlying idea is as follows. It can be viewed as a variation of Schein's original proof of Theorem~\ref{thm:Schein}.
There, he exhibits two specific partial transformations, namely $\alpha_n$ and $\beta_n$ on $[0,n]$. 
The key strategy of Schein's proof is to assume, aiming for contradiction, that  $FI_1$ is finitely presented as a semigroup. 
Schein then fixes an arbitrary finite presentation wih respect to the generators $\{a,a^{-1}\}$, and  shows that, for a sufficiently large $n$, the semigroup $\langle \alpha_n, \beta_n\rangle$ satisfies all these finitely many relations, and yet fails to preserve commutativity of infinitely many 
pairs of idempotents. The essence of our adaptation is to find a map from a finitely generated subsemigroup of $FI_1$ containing infinitely many idempotents to the semigroup $\langle \alpha_n, \beta_n \rangle$. We will show that this map respects multiplication in finitely many $\mathscr{D}$-classes of $FI_1$, but yet again fails to preserve commutativity of infinitely many pairs of idempotents.

We start our proof by characterising when a finitely generated subsemigroup of $FI_1$ has infinitely many idempotents:

\begin{lemma}
\label{lem:useful}
For a subsemigroup $S$ of $FI_1$ generated by a finite set $A \subseteq FI_1$, the following are equivalent:
\begin{enumerate}[leftmargin=9mm,itemsep=1mm,label=\textup{(\roman*)}]
\item \label{useful_i} $S$ contains infinitely many idempotents;
\item \label{useful_ii} $S$ contains both positive and negative elements;
\item \label{useful_iii} $A$ contains both positive and negative elements.
\end{enumerate}
\end{lemma}

\begin{proof}
The equivalence of \ref{useful_ii} and \ref{useful_iii} is clear.

\ref{useful_iii} $\Rightarrow$ \ref{useful_i}. Assume there are elements $u_1 = (-a_1, p_1, b_1)$ and $u_2 = (-a_2, -p_2, b_2)$ in $A$ where $p_1, p_2 > 0$. Let $m$ be a common multiple of $p_1$ and $p_2$ so that $m = n_1p_1 = n_2p_2$ for some $n_1, n_2 \in \mathbb{N}$. A straightforward calculation shows that 
\[
u_1^{n_1x}u_2^{n_2x} = (-\max(a_1, a_2-p_2), 0, \max(n_1xp_1 + (b_1-p_1), n_1xp_1+b_2)) 
\]
gives distinct idempotents for each $x \in \mathbb{N}$.

\ref{useful_i} $\Rightarrow$ \ref{useful_iii}. Suppose $A$ consists entirely of non-positive or non-negative elements. Without loss of generality, we assume the latter. Hence, the product of elements of $A$ is an idempotent if and only if each factor in the product is an idempotent. There are only finitely many idempotents (sign $0$ elements) in $A$, which gives finitely many idempotents in $\langle A \rangle=S$.
\end{proof}

\begin{remark}
\label{rm:fig}
\begin{enumerate}[leftmargin=0.1mm,itemindent=0.8cm,labelwidth=\itemindent,labelsep=0cm,align=left,label={(\arabic*)}]
\item \label{rm_fig_1} A finitely generated subsemigroup of $FI_1$ which contains infinitely many idempotents has two elements of the form $(-a, p, b)$ and $(-c, -p, d)$ where $p > 0$.\smallskip
\item \label{rm_fig_2} \cref{fig:inf_idem} shows a set of idempotents of the form $u_1^{n_1x}u_2^{n_2x}u_2^{n_2y}u_1^{n_1y}$ where $x, y \in \mathbb{N}_0$ with $x+y > 0$. Note that this set is isomorphic again to $\mathbb{N}_0\times\mathbb{N}_0\setminus\{(0,0)\}$ as a partially ordered set. 
Also, there are infinitely many $n\in\mathbb{N}$ such that $D_n \cap \langle u_1, u_2\rangle \neq \emptyset$.\smallskip
\item \label{rm_fig_3} Let $e \in D_n\cap \langle u_1, u_2 \rangle$ be an idempotent of the form $u_1^{n_1x}u_2^{n_2x}u_2^{n_2y}u_1^{n_1y}$. It can be easily seen from \cref{fig:inf_idem} that there exists an idempotent $f \in D_m\cap \langle u_1, u_2 \rangle$ such that $m > n$ and $f$ is incomparable with $e$. More specifically, if $x = 0$, then $f$ can be chosen to be $u_1^{n_1z}u_2^{n_2z}$ with $z > y$. If $x \neq 0$, then we may choose $f$ to be $u_1^{n_1z}u_2^{n_2z}u_2^{n_2t}u_1^{n_1t}$ where $z = x-1$ and $t \geq y+2$.
\end{enumerate}
\end{remark}

\begin{figure}[h]
\centering
\begin{tikzpicture}[scale=0.53]
\node [draw, inner sep=-2, scale=0.2] (f1) at (0.5,8) {};
\node[draw, inner sep=-2, scale=0.2] (e1) at (-0.5,8) {}; 
\foreach \i in {2, ..., 12}{
    \node[draw, inner sep=-2, scale=0.2] (f\i) at ($(f1)+(-60:\i-1)$) {};
    \node[draw, inner sep=-2, scale=0.2] (e\i) at ($(e1)+(-120:\i-1)$) {};
}

\draw (f1)--(f12);
\draw (e1)--(e12);

\foreach \i in {1, ..., 11} {
    \draw (f\i) -- ($(e1)+(-120:11)+(0:\i)$);
    \draw (e\i) -- ($(f1)+(-60:11)-(0:\i)$);
}
\node[draw, shape = circle, fill=black, label=right:{\small $(0,0, \max (b_1-p_1, b_2))$}, inner sep=-2, scale=0.5] (fb0) at (f2) {};
\node[draw, shape = circle, fill=black, label=right:{\small $(0,0,\max(n_{1}p_{1}+(b_{1}-p_{1}), n_{1}p_{1}+b_{2}))$}, inner sep=-2, scale=0.5] (fb1) at (f5) {};

\node[draw, shape = circle, fill=black, label=left:{\small$(-\max(a_{2}-p_{2}, a_{1}),0,0)$}, inner sep=-2, scale=0.5] (ea0) at (e1) {};
\node[draw, shape = circle, fill=black, label=left:{\small $(-\max(n_{2}p_{2}+a_{2}-p_{2}, n_{2}p_{2}+a_{1}),0,0)$}, inner sep=-2, scale=0.5] (ea1) at (e4) {};

\node [draw, shape = circle, fill = red, inner sep=-2.5, scale = 0.8] (u1) at ($(f5)-(60:1)$) {};

\node [draw, shape = circle, fill = red, inner sep=-2.5, scale = 0.8] (u2) at ($(f8)-(60:1)$) {};

\node [draw, shape = circle, fill = red, inner sep=-2.5, scale = 0.8] (u3) at ($(f11)-(60:1)$) {};

\node [draw, shape = circle, fill = red, inner sep=-2.5, scale = 0.8] (u4) at ($(e4)-(120:2)$) {};

\node [draw, shape = circle, fill = red, inner sep=-2.5, scale = 0.8] (u5) at ($(e7)-(120:2)$) {};

\node [draw, shape = circle, fill = red, inner sep=-2.5, scale = 0.8] (u6) at ($(e10)-(120:2)$) {};

\node [draw, shape=circle, fill=red, inner sep=-2.5,scale = 0.8] (u1u'1) at ($(e4)-(120:5)$) {};
\node [draw, shape=circle, fill=red, inner sep=-2.5, scale = 0.8] (u2u'2) at ($(e7)-(120:5)$) {};
\node [draw, shape=circle, fill=red, inner sep=-2.5, scale = 0.8] (u3u'3) at ($(e4)-(120:8)$) {};

\node [draw, shape = circle, fill = black, inner sep=-2, scale=0.2] at ($(e12)+(-120:0.5)$) {};
\node [draw, shape = circle, fill = black, inner sep=-2, scale=0.2] at ($(e12)+(-120:0.75)$) {};
\node [draw, shape = circle, fill = black, inner sep=-2, scale=0.2] at ($(e12)+(-120:1)$) {};

\node [draw, shape = circle, fill = black, inner sep=-2, scale=0.2] at ($(e12)!0.5!(f12) + (-90:0.5)$) {};
\node [draw, shape = circle, fill = black, inner sep=-2, scale=0.2] at ($(e12)!0.5!(f12) + (-90:0.75)$) {};
\node [draw, shape = circle, fill = black, inner sep=-2, scale=0.2] at ($(e12)!0.5!(f12) + (-90:1)$) {};

\node [draw, shape = circle, fill = black, inner sep=-2, scale=0.2] at ($(f12)+(-60:0.5)$) {};
\node [draw, shape = circle, fill = black, inner sep=-2, scale=0.2] at ($(f12)+(-60:0.75)$) {};
\node [draw, shape = circle, fill = black, inner sep=-2, scale=0.2] at ($(f12)+(-60:1)$) {};
\end{tikzpicture}
\caption{The red dots are idempotents of the form $u_1^{n_1x}u_2^{n_2x}u_2^{n_2y}u_1^{n_1y}$ where $x, y \in \mathbb{N}_0$ with $x+y > 0$. The rightmost and leftmost idempotents are $u_1^{n_1x}u_2^{n_2x}$ and $u_2^{n_2y}u_1^{n_1y}$, where $x, y \in \mathbb{N}$, respectively.}
\label{fig:inf_idem}
\end{figure}
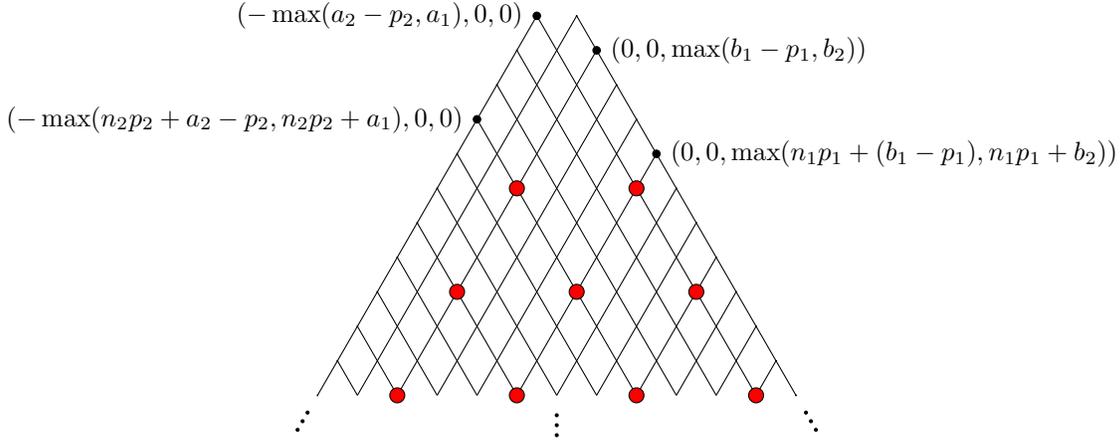

The following lemma introduces two partial transformations $\alpha_n$ and $\beta_n$. It was the key technical observation in \cite{Schein1975} towards proving the main result there (Theorem~\ref{thm:Schein}).

\begin{lemma}[\protect{\cite[Lemma 2]{Schein1975}}]
\label{lem:schein}
For $n\in \mathbb{N}$ define two partial transformations  on $\{0, \dots, n\}$:
\begin{center}
\scalebox{0.95}{$\alpha_n = {\biggl(\begin{array}{ccccccc}
    0 & 1 & 2 & \cdots & n-2 & n-1 & n \\
    1 & 2 & 3 & \cdots &  n-1  & n & -
  \end{array}\biggr),} \ \ \beta_n = { \biggl(\begin{array}{ccccccc}
    0 & 1 & 2 & \cdots & n-2 & n-1 & n \\
    0 & 0 & 1 & \cdots &  n-3  & n-2 & n-1
  \end{array}\biggr)}.$}
\end{center}
Then, the following hold. 
 
\begin{enumerate}[leftmargin=9mm,itemsep=1mm,label=\textup{(\roman*)}]
    \item \label{schein_lem_1} $\alpha_n\beta_n\alpha_n = \alpha_n$;
    \item \label{schein_lem_2} $\beta_n\alpha_n\beta_n = \beta_n$;
    \item \label{schein_lem_3} $\beta_n^{i}\alpha_n^{i}\alpha_n^{j}\beta_n^{j} = \alpha_n^{j}\beta_n^{j}\beta_n^{i}\alpha_n^{i}$\, when $i>n$ or $j>n$ or $i+j \leq n$;
    \item \label{schein_lem_4} $\beta_n^{i}\alpha_n^{i}\alpha_n^{j}\beta_n^{j} \neq \alpha_n^{j}\beta_n^{j}\beta_n^{i}\alpha_n^{i}$\, when $i+j > n$ and $i, j \leq n$. \hfill \qed
\end{enumerate}
\end{lemma}

We will also make use of the following technical lemma concerning presentations and ideals.

\begin{lemma}
\label{lem:ideal}
Let $S$ be a semigroup defined by $\langle A \mid R \rangle$ via $\theta \colon A \to S$. Let $I$ be an ideal of~$S$,  let $C:=S\setminus I$, and let 
\[
A_C := \{a\in A : a\theta \in C\},\quad
R_C := \bigl\{(r, s)\in R : r\theta=s\theta \in C\bigr\}.
\]
Let $T$ be a semigroup defined by $\langle A_C \mid R_C \rangle$ via $\pi \colon A_C\to T$. Suppose $u,v\in A^+$ are such that $u\theta = v\theta \in C$. Then, $u, v \in A_C^{+}$ and $u\pi = v\pi$.
\end{lemma}

\begin{proof}
Consider an elementary sequence 
from $u$ to $v$ with respect to $R$; let $n$ be its length.
That all the words in the sequence are in $A_C^+$ follows from $I$ being an ideal.
We claim that only relations from $R_C$ are used in the sequence.
Inductively, it is sufficient to prove this in the case of sequences of length $1$.
So, suppose $u=prq$, $v=psq$, where $p,q\in A_C^\ast$ and $(r,s)$ or $(s,r)$ is in $R$.
Again, since $I$ is an ideal, we must in fact have $r\theta=s\theta\in C$, and it follows that $u\pi = v\pi$, as required.
\end{proof}

Recall from \cref{pro:FI_1_present} that $FI_1$ is defined by $\langle x,x^{-1}\mid \mathfrak{R} \rangle$ via $\phi \colon  x \mapsto (0, 1, 1),\  x^{-1}\mapsto (-1, -1, 0)$. For $n\in \mathbb{N}$ define a map 
\[
\psi_n \colon \{x, x^{-1}\} \to \langle \alpha_n, \beta_n \rangle,\ 
x\mapsto \alpha_n,\  x^{-1}\mapsto \beta_n.
\]
Note that, of course, $\psi_n$ extends to a homomorphism $\{x,x^{-1}\}^+\rightarrow \langle\alpha_n,\beta_n\rangle$.
However, this homomorphism does not factor through $FI_1$, since $\alpha_n$ and $\beta_n$ do not satisfy all the relations from $\mathfrak{R}$ according to \cref{lem:schein}.
Let $C_n:=\bigcup_{i=1}^{n}D_i$ and $I_n:=\bigcup_{i=n+1}^{\infty}D_i$. 
Clearly, these two sets partition $FI_1$, and, as observed earlier, $I_n$  is an ideal. 
The following is a special case of \cref{lem:ideal}. 

\begin{lemma}
\label{la:special}
Let  the notation be as above, and let $u, v\in \{x, x^{-1}\}^{+}$. If $u\phi = v\phi \in C_n$ then $u\psi_n = v\psi_n$.
\end{lemma}
\begin{proof}
Using the notation from \cref{lem:ideal}, we have
\begin{align*}
A_{C_n} &= \{x, x^{-1}\},\\
\mathfrak{R}_{C_n} &= \bigl\{(r, s)\in \mathfrak{R} : r\phi=s\phi \in C_n\bigr\}\\
& = \bigl\{(xx^{-1}x, x), \, (x^{-1}xx^{-1}, x^{-1})\bigr\}\cup \bigl\{ (x^{-i}x^ix^jx^{-j}, x^{j}x^{-j}x^{-i}x^{i})\colon 0 < i+j \leq n \bigr\}.
\end{align*}
Let $T_n$ be the semigroup defined by $\langle A_{C_n} \mid \mathfrak{R}_{C_n}\rangle$ via $\pi_n\colon A_{C_n} \to T_n$. By \cref{lem:ideal}, we have $u\pi_n = v\pi_n$. By \cref{lem:schein} \ref{schein_lem_1}--\ref{schein_lem_3}, $\langle \alpha_n, \beta_n \rangle$ satisfies the relations in $\mathfrak{R}_{C_n}$.
Hence $\langle \alpha_n, \beta_n \rangle$ is a homomorphic image of $T_n$ with $\mathfrak{R}_{C_n}^{\sharp} \subseteq \ker\psi_n$, and so $u\psi_n = v\psi_n$.
\end{proof}

We will now define a map $\sigma_n\colon C_n\rightarrow \langle \alpha_n, \beta_n \rangle$ for each $n\in \mathbb{N}$ as follows. 
For every $s \in FI_1$ we choose an arbitrary $w_s\in \{x,x^{-1}\}^{+}$ representing $s$, i.e. such that $w_s\phi = s$. We then let $s\sigma_n := w_s\psi_n$. The next lemma plays a significant role in the proof of \cref{prop:forward}. It presents two situations: one in which $\sigma_n$ behaves like a homomorphism by preserving multiplication in finitely many $\mathscr{D}$-classes of $FI_1$, and the other where it does not.

\begin{lemma}
\label{la:new_cor}
Let $m, n\in\mathbb{N}$ be such that $m \geq 3n$. With the above notation, we have:
\begin{enumerate}[leftmargin=9mm,itemsep=1mm,label=\textup{(\roman*)}]
\item \label{new_cor_1} If $s_1, s_2\in C_n$ then $s_1s_2\in C_m$ and $(s_1s_2)\sigma_m = (s_1\sigma_m )(s_2\sigma_m)$.
\item \label{new_cor_2} If $e \in C_n$ and $f \in D_m$ are incomparable idempotents then $(e\sigma_m)(f\sigma_m) \neq (f\sigma_m)(e\sigma_m)$.
\end{enumerate}
\end{lemma}

\begin{proof}
For brevity we write $\alpha, \beta, \psi, \sigma$ for $\alpha_m, \beta_m, \psi_m, \sigma_m$, respectively. 
\smallskip

\ref{new_cor_1} Write $s_1=(-a_1, p_1, b_1)$ and $s_2=(-a_2, p_2, b_2)$ where $a_1+b_1, a_2+b_2 \leq n$. By definition,
\[
s_1s_2 = (-\max(a_1, a_2-p_1), p_1+p_2, \max(b_1, b_2+p_1)).
\]
Since 
\[
\max(a_1, a_2-p_1)+\max(b_1, b_2+p_1)\leq 3n \leq m,
\]
it follows that $s_1s_2 \in C_m$.

We now  show that $(s_1s_2)\sigma = (s_1\sigma)(s_2\sigma)$.
By the definition of  $\sigma$ we have
\[
(s_1s_2)\sigma = (w_{s_1s_2})\psi\quad\text{and}\quad
(s_1\sigma)(s_2\sigma) = (w_{s_1}\psi)(w_{s_2}\psi) = (w_{s_1}w_{s_2})\psi.
\]
Note that $w_{s_1s_2}\phi = s_1s_2 = (w_{s_1}w_{s_2})\phi$.
Since $s_1s_2 \in C_m$, \cref{la:special} gives $w_{s_1s_2}\psi = (w_{s_1}w_{s_2})\psi$, and the result follows.
\smallskip

\ref{new_cor_2} 
Let $e = (-a, 0, b) $ and $f = (-c, 0, d)$.
The assumptions $e\in C_n$, $f\in D_m$ and incomparability of $e$ and $f$ imply
$a+b \leq n$, $c+d=m$ and 
$(a-c)(b-d)<0$. 
Without loss of generality, we assume that
$a > c$ and $b < d$.
Note that $w_e\phi = e = (x^{-a}x^ax^bx^{-b})\phi$, $w_f\phi = f = (x^{-c}x^cx^dx^{-d})\phi$.
Using $e, f\in C_m$ and 
\cref{la:special} again, we deduce that 
\[
e\sigma=w_e\psi=(x^{-a}x^ax^bx^{-b})\psi=\beta^a\alpha^a\alpha^b\beta^b\quad
\text{and similarly} \quad
f\sigma=\beta^c\alpha^c\alpha^d\beta^d.
\]
Now we consider the mappings $(e\sigma)(f\sigma)$ and $(f\sigma)(e\sigma)$. 
We compute the following:
\begin{align*}
(e\sigma)(f\sigma) &= \beta^a\alpha^a\alpha^b\beta^b\beta^c\alpha^c\alpha^d\beta^d\\
&=\beta^a\alpha^a\beta^c\alpha^c\alpha^b\beta^b\alpha^d\beta^d && \text{by \cref{lem:schein} \ref{schein_lem_3}
since } b+c\leq m\\
&=\beta^a\alpha^a\alpha^d\beta^d&& \text{since $\alpha^c\beta^c\alpha^c = \alpha^c$ and $\alpha^b\beta^b\alpha^b = \alpha^b$,}\\
(f\sigma)(e\sigma) &= \beta^c\alpha^c\alpha^d\beta^d\beta^a\alpha^a\alpha^b\beta^b\\
&=\alpha^d\beta^d\alpha^b\beta^b\beta^c\alpha^c\beta^a\alpha^a &&\text{by \cref{lem:schein} \ref{schein_lem_3} since } c+d,a+b,b+c\leq m\\
&=\alpha^d\beta^d\beta^a\alpha^a&&\text{since $\beta^b\alpha^b\beta^b = \beta^b$ and $\beta^c\alpha^c\beta^c = \beta^c$.}
\end{align*}
Since $a+d > c+d = m$, $(e\sigma)(f\sigma) \neq (f\sigma)(e\sigma)$ by \cref{lem:schein} \ref{schein_lem_4}.
\end{proof}

We are now ready to prove \cref{prop:forward}. 
A key idea is that if a finitely generated subsemigroup of $FI_1$ has infinitely many idempotents, then we can find large enough natural numbers $n$ and $m$ as in \cref{la:new_cor}.

\begin{proof}[Proof of \cref{prop:forward}]
Let $S$ be a finitely generated subsemigroup of $FI_1$ which contains infinitely many idempotents. Aiming for a contradiction, suppose that $S$ is defined by a finite presentation $\langle A \mid R \rangle$ via $\theta \colon A \to S$. Let $\eta \colon A^+ \to \{x, x^{-1}\}^+$ be the homomorphism defined by $a\eta = w_{a\theta}$ for $a\in A$. Note that $a\eta\phi = w_{a\theta}\phi = a\theta$, which shows $\eta\phi = \theta$.

By \cref{lem:useful}, $S$ contains both positive and negative elements. Recall that $S$ particularly contains some $u_1 = (-a_1, p, b_1)$ and $u_2 = (-a_2, -p, b_2)$ with $p > 0$ (\cref{rm:fig} \ref{rm_fig_1}). Let $n\in \mathbb{N}$ be such that $C_n$ contains all of the following elements:
\begin{itemize}
    \item an idempotent of the form $u_1^xu_2^xu_2^yu_1^y$ where $x,y\in \mathbb{N}_0$ with $x+y > 0$;
    \item $a\theta$ for each $a\in A$;
    \item $r\theta\,\, (= s\theta)$ for each $(r,s)\in R$.
\end{itemize}
Let $m\in \mathbb{N}$ be such that $m \geq 3n$ and $D_m$ contains an idempotent of the form $u_1^xu_2^xu_2^yu_1^y$. Recall the map $\sigma_m \colon C_m \to \langle \alpha_m, \beta_m \rangle$ and the natural homomorphism $\psi_m\colon \{x, x^{-1}\}^+ \to \langle \alpha_m, \beta_m \rangle$. \cref{fig:com_diag_1} shows the diagram of maps and homomorphisms we have defined.

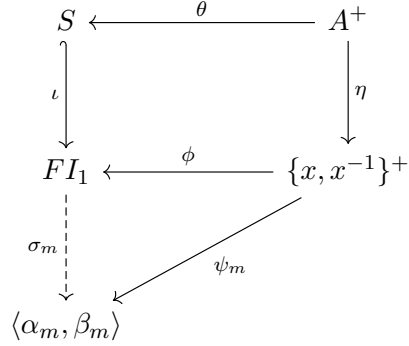
\begin{figure}[H]
\centering
\begin{tikzcd}[row sep=huge, column sep=huge]
S \arrow[d, hook', "\iota", swap]
& A^{+} \arrow[l, "\theta", sloped, above]\arrow[d, "\eta"] \\
FI_1 \arrow[d, dashed, "\sigma_m", swap]
& \{x,x^{-1}\}^{+} \arrow[l, "\phi", sloped, above] \arrow[dl, "\psi_m"] \\
\langle \alpha_m, \beta_m \rangle
\end{tikzcd}
\caption{Homomorphisms and maps in the proof of \cref{prop:forward}. Here $\iota$ denotes the inclusion mapping. The upper part of the diagram is commutative, i.e. $\theta\iota = \eta\phi$. Note that the domain of $\sigma_m$ is $C_m \subseteq FI_1$.}
\label{fig:com_diag_1}
\end{figure}

From now on, we take $m$ and $n$ to be fixed and we write $\alpha, \beta, \psi$ and $\sigma$ for $\alpha_m, \beta_m, \psi_m$ and $\sigma_m$, respectively. We claim that 
\begin{equation}
    \ker\theta \subseteq \ker\eta\psi. \label{eq:ker}
\end{equation}
Suppose $(r, s)\in R$. Note that $r\eta\phi = (w_{r\theta})\phi = r\theta = s\theta = (w_{s\theta})\phi = s\eta\phi$. Moreover, $r\theta\,(=s\theta) \in C_n\subseteq C_m$. Hence, $r\eta\psi = s\eta\psi$ by \cref{la:special} and so indeed we have (\ref{eq:ker}).

By our construction, there are idempotents of the form $u_1^xu_2^xu_2^yu_1^y$ in $C_n\cap S$ and $D_m\cap S$. 
Pick two idempotents $e$ and $f$ of this form such that $e\in C_n\cap S$, $f\in D_m\cap S$ and $e$ and $f$ are incomparable
 (see \cref{rm:fig} \ref{rm_fig_3}). 
 Now, let $v_e,v_f\in A^+$ be such that $v_e\theta = e$ and $v_f\theta = f$. Since $ef = fe$ in $S$, we have 
 \[
(v_ev_f)\theta= (v_e\theta)(v_f\theta) = ef=fe=  (v_f\theta)(v_e\theta)=(v_fv_e)\theta,
 \] 
i.e.\ $(v_ev_f, v_fv_e)\in \ker\theta$. Therefore, $(v_ev_f)\eta\psi = (v_fv_e)\eta\psi$ by (\ref{eq:ker}).

Note that we have $v_e\eta\phi = v_e\theta = e = w_e\phi$ and $e\in C_m$. By \cref{la:special}, $v_e\eta\psi = w_e\psi$. Similarly, $v_f\eta\psi = w_f\psi$. Then,
\[
(v_ev_f)\eta\psi = (v_e\eta\psi)(v_f\eta\psi) = (w_e\psi)(w_f\psi) = (e\sigma)(f\sigma),
\]
and, analogously,
$(v_fv_e)\eta\psi = (f\sigma)(e\sigma)$. However, $(e\sigma)(f\sigma) \neq (f\sigma)(e\sigma)$ by \cref{la:new_cor}~\ref{new_cor_2}. We have obtained a contradiction from the assumption that $S$ is finitely presented, and hence the proposition follows.
\end{proof}

\section{Finitely many idempotents}
\label{sec:backward}

In this section, we prove the following proposition and thus complete the proof of the \nameref{thm:main}.

\begin{proposition}
\label{prop:backward}
Let $S$ be a finitely generated subsemigroup of $FI_1$. If $S$ contains finitely many idempotents then $S$ is finitely presented.
\end{proposition}

In our proof we will make use of the following three results:

\begin{proposition}[{\cite[Theorems 1.1, 1.3]{nik_large}}]
\label{prop:nik_large}
Let $T$ be a subsemigroup of a semigroup $S$ such that $S\setminus T$ is finite. Then $S$ is finitely generated (resp.\ presented) if and only if $T$ is finitely generated (resp.\ presented). \qqed
\end{proposition}

An equivalence relation $\rho$ on a semigroup $S$ is called a \emph{congruence} if it is compatible with the multiplication,
i.e.\ if $(x/\rho)(y/\rho)\subseteq (xy)/\rho$ for all $x,y\in S$; see \cite[Section 1.5]{Howie}.
The \emph{index} of an equivalence relation is the number of its equivalence classes.

\begin{proposition}[{\cite[Theorem 4.1, Corollary 4.5]{union}}]
\label{prop:union}
Let $S$ be a semigroup and let $\rho$ be a congruence relation on $S$ of finite index such that every equivalence class is a subsemigroup of $S$. If each congruence class is finitely presented then so is $S$. \qqed
\end{proposition}

The following proposition has already been discussed in the \nameref{sec:intro}.

\begin{proposition}
\label{prop:sub_N}
Any subsemigroup of $\mathbb{N}$ is finitely generated and finitely presented.\qqed
\end{proposition}

To prove \cref{prop:backward} we first consider a finitely generated subsemigroup $S$ of $FI_1$ wholly consisting of positive elements, i.e.
\[
S\leq P:= \bigl\{(-a, p, b)\in FI_1 \colon p>0 \bigr\} \: (\leq FI_1).
\]
For such an $S$ we will establish a decomposition into a disjoint union of subsemigroups of $\mathbb{N}$, in such a way that Propositions \ref{prop:union} and \ref{prop:sub_N} can be applied. Finally an appeal to \cref{prop:nik_large} and \cref{lem:useful} will enable us to extend to an arbitrary $S$.

We begin to implement this programme. The first two lemmas discuss a decomposition of~$P$ into infinitely many copies of $\mathbb{N}$.

\begin{lemma}
\label{la:decompose}
For $x \in \mathbb{N}_0$ and $y\in \mathbb{N}$ define
\[
N_{x,y} :=  \bigl\{(-x, p, p+(y-1)) \colon p\in \mathbb{N}\bigr \}.
\]
\begin{enumerate}[leftmargin=9mm,itemsep=1mm,label=\textup{(\roman*)}]
\item\label{it:dec1}
Each $N_{x,y}$ is a subsemigroup of $P$ isomorphic to $\mathbb{N}$.
\item\label{it:dec2}
The subsemigroups $N_{x,y}$ are pairwise disjoint.
\item\label{it:dec3}
$P=\bigcup\bigl\{ N_{x,y}\colon x\in\mathbb{N}_0,\ y\in\mathbb{N}\bigr\}$.
\end{enumerate}
\end{lemma}

\begin{proof}
\ref{it:dec1}
follows by observing that $N_{x,y}$ is generated by $(-x,1,y)$, while 
\ref{it:dec2}
and \ref{it:dec3}
are obvious.
\end{proof}

Note that the last component of an element of $N_{x,y}$ is the sum of the middle component and $y-1$. To emphasise this we will sometimes write $(-x, p, p+(y-p))$ for $(-x, p, y)$ in what follows.

\begin{lemma}
\label{lem:chain}
For $n\in \mathbb{N}_0$ define
\[
T_n:=\bigcup\bigl\{N_{x,y} \colon 0 \leq x \leq n,\, 1\leq y \leq n+1\bigr\}.
\]
\begin{enumerate}[leftmargin=9mm,itemsep=1mm,label=\textup{(\roman*)}]
\item\label{it:cha1}
Each $T_n$ is a subsemigroup of $P$.
\item\label{it:cha2}
$T_0\leq T_1\leq T_2\leq\dots$. 
\item\label{it:cha3}
$P=\bigcup_{n\in\mathbb{N}_0} T_n$.
\end{enumerate}
\end{lemma}

\begin{proof}
\ref{it:cha1}
Note that 
\[
(-a, p, b)  \in T_n\quad \Leftrightarrow\quad p>0\text{ and } a,b-p \in [0,n].
\]
Suppose
$(-a, p, b), (-c, q, d) \in T_n$. Then
\[
(-a, p, b)(-c, q, d)
=(-\max(a, c-p), p+q, \max(b, d+p)).
\]
Since $a,c\in [0,n]$ it follows that $\max(a, c-p) \in [0,n]$ as well.
Next, from $b-p,d-q\in [0,n]$ it follows that $b-(p+q)\leq n$ and
 $(d+p)-(p+q)=d-q\in [0,n]$. Hence $\max(b,d+p)-(p+q)\in [0,n]$.
Since  $p+q>0$, it follows that $(-a, p, b)(-c, q, d)\in T_n$, and therefore $T_n\leq P$.

\ref{it:cha2} is obvious, and \ref{it:cha3} follows from \cref{la:decompose} \ref{it:dec3}.
\end{proof}

We now turn to investigate a finitely generated subsemigroup of $P$. The first result follows easily from the previous lemma:

\begin{lemma}
\label{la:fini} 
A finitely generated subsemigroup $S$ of $P$ intersects only finitely many $N_{x,y}$.
\end{lemma}

\begin{proof}
As $S$ is finitely generated, there must exist $n\in\mathbb{N}_0$ such that $S\subseteq T_n$ because of \cref{lem:chain} \ref{it:cha2}, \ref{it:cha3}. But $T_n$ in turn is a union of finitely many $N_{x,y}$ by definition, and the result follows.
\end{proof}

As a key technical step, we will now show that a finitely generated subsemigroup $S$ of $P$ with an additional technical condition is finitely presented. 

\begin{lemma}
\label{lem:nice}
Let $S \leq P$ be generated by $A = \bigl\{(-a_1, p_1, b_1), \dots, (-a_n, p_n, b_n)\bigr\}$. Suppose that for all $i, j \in [n]$ we have 
\begin{equation} 
\label{eq:nice_cond}
a_i \geq a_j - p_i \;\;\text{ and }\;\; b_i \leq b_j + p_i.
\end{equation}
For each $(x,y)\in I:=\{a_1, \dots, a_n\}\times \{b_1-p_1, \dots, b_n-p_n\}$, define
\[
S_{x,y} := N_{x,y+1}\cap S.
\]
Then the following hold.
\begin{enumerate}[leftmargin=9mm,itemsep=1mm,label=\textup{(\roman*)}]
\item \label{nice_1a} The sets $S_{x,y}$, $(x,y)\in I$, partition $S$.
\item \label{nice_1b} The equivalence relation with the equivalence classes $S_{x,y}$, $(x,y)\in I$, is a congruence on $S$.
\item \label{nice_2} $S$ is finitely presented.
\end{enumerate}
\end{lemma}

\begin{proof}
We first prove the following:

\begin{namedclaim}[Claim]
\label{cl:prods}
We have 
\[
(-c_1, q_1, d_1)\dots(-c_m, q_m, d_m) = \bigl(-c_1, q_1+\dots+q_m, q_1+\dots+q_m+(d_m-q_m)\bigr)\in S_{c_1,d_m-q_m},
\]
where each $(-c_i, q_i, d_i)\in A$.
\end{namedclaim}

\begin{proof}
We prove the claim by induction on $m$. The case $m=1$ is trivial.
Consider $m>1$, and inductively assume the claim is valid for $m-1$. 
Then
\begin{align*}
&(-c_1, q_1, d_1)\dots(-c_{m}, q_{m}, d_{m})\\
=&\bigl((-c_1, q_1+\dots+q_{m-1}, q_1+\dots+q_{m-1}+(d_{m-1}-q_{m-1})\bigr)(-c_{m}, q_{m}, d_{m})=(-c, q, d),
\end{align*}
where
\begin{align*}
c&=\max\bigl(c_1, c_{m}-(q_1+\dots+q_{m-1})\bigr),\\
 q&=q_1+\dots+ q_m,\\
d&=\max\bigl(q_1+\dots+q_{m-1}+(d_{m-1}-q_{m-1}), q_1+\dots+q_{m-1}+d_{m}\bigr).
\end{align*}
Using (\ref{eq:nice_cond}) we have
\[
 c_1 \geq c_{m}-q_1 \geq c_{m}-(q_1+\dots+q_m),
 \]
  and so $c=c_1$.
Similarly, 
\[
q_1+\dots +q_{m-1}+(d_{m-1}-q_{m-1})\leq 
q_1+\dots +q_{m-1}+d_m=q_1+\dots+q_m+(d_m-q_m),
\]
implying $d=q_1+\dots+q_m+(d_m-q_m)$. This completes the inductive step and proof of the claim.
\end{proof}
\ref{nice_1a}
First note that $S= \bigcup_{(x,y)\in I}S_{x,y}$: indeed, ($\subseteq$) follows from \nameref{cl:prods}, whereas ($\supseteq$) holds by the definition of the $S_{x,y}$.
That each $S_{x,y}$ is non-empty also follows from \nameref{cl:prods}: indeed, products of elements of $A$ of length $2$ already yield members in each $S_{x,y}$. Finally, that the $S_{x,y}$ are mutually disjoint follows from 
\cref{la:decompose} \ref{it:dec2}.

\ref{nice_1b}
\nameref{cl:prods} gives that $S_{x,y}S_{z,t}\subseteq S_{x,t}$ for all $(x,y),(z,t)\in I$, which readily implies this part.

\ref{nice_2} Each $S_{x,y}$ is a subsemigroup of $N_{x,y+1}$, and $N_{x,y+1}\cong \mathbb{N}$ by 
\cref{la:decompose}~\ref{it:dec1}.
Hence each $S_{x,y}$ is finitely presented by \cref{prop:sub_N}.
It now follows that $S$ is finitely presented by combining \ref{nice_1a}, \ref{nice_1b} and \cref{prop:union}.
\end{proof}

Next we extend to an arbitrary finitely generated subsemigroup of $P$.

\begin{lemma}
\label{la:subP}
Any finitely generated subsemigroup $S$ of $P$ is finitely presented.
\end{lemma}

\begin{proof}
By \cref{la:fini} there are only finitely many $N_{x,y}$ which have non-empty intersection with $S$; suppose they are $N_{x_1,y_1}, \dots , N_{x_m, y_m}$. For each $i = 1, \dots, m$, let 
\[
q_i := \max(x_i, y_i-1),\quad Q := \max_{1\leq i \leq m}q_i,\quad
U_i := \bigl\{(-a, p, b)\in N_{x_i,y_i} : p > Q\bigr\} \cap S.
\]
Since the multiplication in $P$ always increases the middle component, it follows that $U_i \leq S \cap N_{x_i,y_i}$ and that $U := \bigcup_{i=1}^{m}U_i\leq S$. Since $U_i \leq N_{x_i,y_i} \cong \mathbb{N}$, it follows that each $U_i$ is finitely generated by \cref{prop:sub_N}. Hence, $U$ is also finitely generated and we let $A = \{(-a_1, p_1, b_1), \dots, (-a_n, p_n, b_n)\}\subseteq P$ be its generating set.

We claim that arbitrary $(-a_i, p_i, b_i), (-a_j, p_j, b_j) \in A$ satisfy the condition (\ref{eq:nice_cond}) in \cref{lem:nice}. We have $p_i, p_j > Q$, $a_i, a_j \in \{x_1, \dots, x_m\}$ and $b_i-p_i, b_j-p_j \in \{y_1-1, \dots , y_m-1\}$. By our choice of $Q$, we have $a_j \leq Q < p_i$ which implies $a_j - p_i < 0 \leq a_i$. Also, $b_i - p_i \leq Q < p_j \leq b_j$ which implies $b_i \leq b_j + p_i$, and (\ref{eq:nice_cond}) indeed holds. 

\cref{lem:nice} \ref{nice_2} now implies that $U$ is finitely presented. Since $S\setminus U$ is finite, $S$ is also finitely presented by \cref{prop:nik_large}.
\end{proof}




Finally, we can prove the main result of this section.

\begin{proof}[Proof of \cref{prop:backward}]
Let $S$ be a finitely generated subsemigroup of $FI_1$ with only finitely many idempotents.
By \cref{lem:useful}, $S$ consists entirely of non-negative  or   non-positive elements.
Without loss of generality suppose the former is the case.
Notice that $S\cap P$ is an ideal of $S$, and that $S\setminus P$ is the set of idempotents of $S$.
Hence $S\setminus (S\cap P)$ is finite.
By \cref{prop:nik_large}, $S\cap P$ is finitely generated. Then it follows that it is finitely presented by \cref{la:subP}. And finally another application of \cref{prop:nik_large} gives that $S$ itself is finitely presented, completing the proof.
\end{proof}

The \nameref{cor:inv_fp} can be quickly derived from the \nameref{thm:main}:

\begin{proof}[Proof of the \nameref{cor:inv_fp}]
If $S$ contains a non-idempotent element $a$, it also contains infinitely many idempotents $a^{-n}a^n$, $n\in\mathbb{N}$, so $S$ cannot be finitely presented by our \nameref{thm:main}.
\end{proof}

We conclude the paper by recalling the results from \cite{CS:Rcf01} and \cite{GS:Fim21}, and ask whether they extend to subsemigroups of the monogenic free inverse semigroup, in the sense that our \nameref{thm:main} extends the original theorem by Schein (Theorem \ref{thm:Schein}). Specifically, we ask whether it is true that each of the following two properties concerning a finitely generated subsemigroup $S$ of $FI_1$ is equivalent to $S$ having only finitely many idempotents: (a) having a regular language of normal forms; and (b) the property $\textup{FP}_2$?
\vspace{2mm}

\textbf{Competing interest declaration.}
The authors have no competing interest to declare.


\bibliographystyle{plainnat}

\end{document}